\documentclass[11pt, a4paper, reqno]{amsart}
\usepackage{amsmath,amsfonts,amssymb,amsthm,color,esint}

\newcommand{\R}{\mathbb{R}}

\newcommand{\N}{\mathbb{N}}

\newcommand{\Om}{\Omega}
\newcommand{\pip}{\varphi}
\newcommand{\homB}{HB^\theta_{p,p}}
\newcommand{\homb}{\homB(\partial\Om)}
\newcommand{\inB}{B^\theta_{p,p}}
\newcommand{\inb}{\inB(\partial\Om)}
\newcommand{\dotb}{\dot{B}^\theta_{p,p}}
\newcommand{\homn}{HN^{1,p}(\Om)}
\newcommand{\inn}{N^{1,p}(\Om)}

\DeclareMathOperator*{\diam}{diam}
\DeclareMathOperator*{\diver}{div}
\DeclareMathOperator*{\dist}{dist}

\DeclareMathOperator*{\rad}{rad}
\DeclareMathOperator*{\Tr}{Tr}
\DeclareMathOperator*{\DtN}{DtN}
\DeclareMathOperator*{\NtD}{NtD}
\DeclareMathOperator*{\dtn}{DtN}
\DeclareMathOperator*{\ntd}{NtD}
\DeclareMathOperator*{\elimsup}{ess\,lim\,sup}

\newcommand{\capa}{\text{Cap}_p}

\newcommand{\loc}{\mathrm{loc}}
\newcommand{\eps}{\varepsilon}

\theoremstyle{plain}
\newtheorem{theorem}[equation]{Theorem}
\newtheorem*{introtheorem}{Theorem}
\newtheorem{lemma}[equation]{Lemma}
\newtheorem{corollary}[equation]{Corollary}

\numberwithin{equation}{section}

\theoremstyle{definition}
\newtheorem{definition}[equation]{Definition}
\newtheorem{standing}[equation]{Standing Assumptions}

\theoremstyle{remark}
\newtheorem{remark}[equation]{Remark}

\begin{document}

\title[Dirichlet-to-Neumann for metric measure spaces]{On the Dirichlet-to-Neumann Map for the $p$-Laplacian on a Metric Measure Space}


\author{Ryan Gibara}
\address{Department of Mathematical Sciences, P.O. Box 210025, University of
Cincinnati, Cincinnati, OH 45221--0025, U.S.A.}
\email{gibararn@ucmail.uc.edu}

\author{Nageswari Shanmugalingam}
\address{Department of Mathematical Sciences, P.O. Box 210025, University of
Cincinnati, Cincinnati, OH 45221--0025, U.S.A.}
\email{shanmun@uc.edu}
\thanks{N.S.'s research was partially supported by NSF (US)
grant DMS\#2054960.}

\date{\today}

\keywords{boundary-value problems, Besov space, Newton-Sobolev space, doubling measure, 
homogeneous spaces, Poincar\'e inequality,
upper gradients, Cheeger differential structure, Dirichlet-to-Neumann map.}

\subjclass[2020]{Primary: 31E05; Secondary: 46E36, 31B20, 45Q05.}

\begin{abstract}
In this note, we construct a Dirichlet-to-Neumann map, from a Besov space of functions, to the dual of
this class. The Besov spaces are of functions on the boundary of a bounded, locally compact uniform domain equipped with
a doubling measure supporting a $p$-Poincar\'e inequality so that this boundary is also equipped with a Radon measure that
has a codimensional relationship with the measure on the domain. We construct this map via the following recipe.
We show first that solutions to Dirichlet problem for the $p$-Laplacian 
on the domain with prescribed boundary data in the Besov space
induce an operator that lives in the dual of the Besov space. 
Conversely, we show that there is a solution, in the homogeneous Newton-Sobolev space, 
to the Neumann problem for the $p$-Laplacian with the Neumann boundary data given by a continuous linear functional 
belonging to the dual of the Besov space. We also obtain bounds on its operator norm in terms of the norms of trace and 
extension operators that relate Newton-Sobolev functions on the domain to Besov functions on the boundary.
\end{abstract}

\maketitle


\section{Introduction}

Classically, for a sufficiently smooth bounded domain $\Omega$ in $\R^n$, the Dirichlet-to-Neumann map 
for the Laplacian is $\DtN:H^{1/2}(\partial\Omega)\rightarrow H^{-1/2}(\partial\Omega)$ 
given by $\DtN(f):=\partial_\eta u:=\eta\cdot\nabla u$, where $\eta$ is the unit outer normal vector to the 
domain $\Omega$ and $u$ is the unique harmonic extension of $f$ from the boundary $\partial\Om$ to $\Omega$. Here, 
$H^{1/2}(\partial\Omega)$ is the classical fractional Sobolev space, with $\partial\Omega$ equipped 
with its $(n-1)$-dimensional Hausdorff measure, and $H^{-1/2}(\partial\Omega)$ is its dual. For more 
general domains, where $\eta$ may not be well defined and $\nabla u$ may not exist pointwise, one may 
define the outer normal derivative of $u\in H^1(\Omega)$ in the weak sense that 
\[
\int_{\Omega}\! (\Delta u)(v) +\int_{\Omega}\!\nabla u \cdot \nabla v = \int_{\partial\Omega}\! 
(\partial_\eta u)(\Tr v) 
\]
for all $v\in H^1(\Omega)$. Here $\partial_\eta u$ should be thought of as a function on $\partial\Om$
that serves the role of a weak outer normal derivative of $u$ at $\partial\Om$; such function may not always exist.
It should be clear that the role of $\partial_\eta u$ in the above is 
reminiscent of that of the usual outer normal derivative in Green's formula. The corresponding 
Dirichlet-to-Neumann map has been studied for other elliptic operators such as the divergence-form 
operator $L=\diver(\gamma\nabla u)$ for a smooth function $\gamma$, the Helmholtz operator, or Schr\"odinger operator,
see \cite{AE, CM, FKOU, Tay}. The study of the Dirichlet-to-Neumann map has also been extended 
to the setting of Riemannian manifolds for the Laplace-Beltrami operator and 
generalized to differential forms \cite{CGGS,GKLP}. 
The recent paper~\cite{Daners2} gave a construction 
of a Dirichlet-to-Neumann map for the nonlinear operator $\Delta_p$, i.e.~the $p$-Laplacian, in the setting of
bounded Lipschitz domains in Euclidean spaces.

The goal of the present paper is to generalize this problem even further to the setting of a 
metric measure space $(\Omega, d,\mu)$ that is 
bounded, locally compact, supporting a doubling measure $\mu$ and satisfies a $p$-Poincar\'e inequality for a 
fixed $1<p<\infty$, so that $\Om$ is also a uniform domain in its completion $\overline{\Om}$. 
We consider Cheeger-type operators  that are analogs of the $p$-Laplacian, 
and endow the boundary $\partial\Omega$ with a measure $\nu$ that is in a codimensional relationship with $\mu$. 

The first main result of this note, contained within Theorem~\ref{th:2}, is the following, in which we demonstrate the 
existence of a solution in the homogeneous Newton-Sobolev space $HN^{1,p}(\Omega)$ to the Neumann 
problem with boundary data taken from the dual of the homogeneous Besov space 
$HB^\theta_{p,p}(\partial\Omega)$. Here, $\theta$ depends on $p$ and the 
codimensionality of the boundary measure $\nu$.

\begin{introtheorem}
Given an operator $L\!\in\!(HB^\theta_{p,p}(\partial\Omega))^*$, there exists $u_L\!\in\! HN^{1,p}(\Omega)$ satisfying
\[
\int_{\Omega}\!|\nabla u_L|^{p-2}\langle \nabla u_L, \nabla v\rangle\,d\mu = L(\Tr(v))
\]
for all $v\in HN^{1,p}(\Omega)$. 
\end{introtheorem}

In the setting of this paper, results of Mal\'y \cite{Mal} imply that there exist a bounded linear trace 
operator $\Tr:HN^{1,p}(\Omega)\rightarrow HB^\theta_{p,p}(\partial\Omega)$ and a bounded linear extension 
operator $E:HB^\theta_{p,p}(\partial\Omega)\rightarrow HN^{1,p}(\Omega)$ with $\Tr\circ E$
the identity operator. We denote by $\|\!\Tr\!\|$ and $\|E\|$ 
their respective operator norms. 

The second main result of the this note, see Corollaries~\ref{cor:DtN-norm} and~\ref{cor:NtD-norm} and 
Theorem~\ref{th:inverse}, is the existence of the Dirichlet-to-Neumann map and its inverse, the Neumann-to-Dirichlet 
map, in this setting. For $f\in \homb$, we denote by $u_f$ is the solution to the Dirichlet problem with boundary data $f$. 
This allows us to create a functional $L_f\in (\homb)^*$ given by 
\[
L_f(g):=\int_{\Omega}\!|\nabla u_f|^{p-2}\langle \nabla u_f,\nabla Eg\rangle\,d\mu 
\]
for any $g\in \homb$. 

\begin{introtheorem}
The Dirichlet-to-Neumann map $\DtN(f):=L_f$ and the Neumann-to-Dirichlet map $\NtD(L):=\Tr(u_L)$ are well-defined 
and bounded operators. Moreover, they are inverses of one another and the operator norm of $\DtN$ satisfies
\[
c_p\ \|\Tr\|^{\tfrac{p}{1-p}}\le \|DtN\|\le \|E\|^p,
\]
where
\[
c_p:=\left[(2p)^{\tfrac{1}{p-1}}+\frac{2(p-1)}{p}\right]^{-1}.
\]
\end{introtheorem}

The Dirichlet-to-Neumann operator plays an important role in the theory of inverse problems and 
has real-world applications. In particular, in medical imaging, one is interested in obtaining 
tomographic representations of the inside of organs by measuring the electrical impedance of a 
current traveling through the tissue of the organ. With knowledge of the Dirichlet-to-Neumann map, 
also called the voltage-to-current map, this can be accomplished non-invasively. 
For a survey on inverse problems and applications, we refer 
the reader to~\cite{I, Uhl}. The generalization of the operator to the setting of metric measure
spaces permit us to then consider applications in situations where we do not have access to 
smooth structures, s for example in the case where the ambient material corresponds to highly nonsmooth
weights
modifying measures on Euclidean domain, and also allow us to consider structures in the non-manifold
setting of Carnot-Carath\'eodory spaces.

The construction of the Dirichlet-to-Neumann map, as described in this note, while developed independently 
of~\cite{Daners2}, follows a similar method as~\cite{Daners2}, but avoids the need for the domain to be Lipschitz
regular.


\section{Background}

\subsection{Setting}
In this section, we will review the necessary definitions related to the analysis on metric spaces and set the 
notation for the rest of this note. Let $(\Omega,d)$ be a locally compact, non-complete metric space with 
completion $\overline{\Om}$ and boundary $\partial\Om:=\overline{\Om}\setminus\Om$. 

In the study of boundary-value problems on such a space, one needs to be able to make sense of notions 
such as trace to the boundary. At this level of generality, where there is no notion of smoothness in the classical 
sense, one must rely on geometric conditions instead. 
In this paper we require the geometric condition called uniform domain condition.
First introduced by Martio and Sarvas in \cite{MS}, 
uniform domains are studied extensively in relation to quasiconformal maps \cite{MS, GeOs, ZP, ZHPL, LPZ, BHK}, potential 
theory~\cite{BS, KT, BBS,GKS,GS}, and extension domains \cite{Jones, Jones2, HerKo, BS, BD}, and the references therein. 
The class of uniform domains includes all bounded Euclidean Lipschitz domains and cardioid domains, but also domains with 
non-rectifiable boundaries such as the von Koch snowflake domain.

\begin{definition}
Given a number $C\ge 1$ and points $x,y\in\Om$, we say that a curve $\gamma$ with end points $x$ and $y$ is
said to be a $C$-\emph{uniform curve} connecting $x$ to $y$ if the length $\ell(\gamma)$ of $\gamma$ satisfies
the condition $\ell(\gamma)\le C\, d(x,y)$ and in addition, for each $z$ in the trajectory of $\gamma$
and for each pair of subcurves $\gamma_{x,z}$ and $\gamma_{z,y}$ of $\gamma$ with end points $x,z$ and $z,y$ respectively,
we have that
\[
\min\{\ell(\gamma_{x,z}),\ell(\gamma_{z,y})\}\le C\, \dist(z,\partial\Om).
\]
The latter condition above is also known as a twisted double cone condition in some literature.
We say that $\Omega$ is a \emph{uniform domain} if there is a constant $C\ge 1$ for which each pair 
$x,y\in\Omega$ with $x\neq y$ can be connected by a $C$-uniform curve $\gamma$. 
\end{definition}

A triple $(\Omega, d, \mu)$ is called a \emph{metric measure space} if $\mu$ is a Borel regular 
measure satisfying $0<\mu(B)<\infty$ for all balls $B\subset\Omega$. We will say that $\mu$ is 
\emph{doubling} if there exists a constant $C\geq 1$ for which $\mu(2B)\leq C \,\mu(B)$ holds for 
all balls $B\subset\Omega$, where $2B$ denotes the ball concentric with $B$ having radius 
twice that of $B$. For references on analysis on metric spaces, we refer the readers to \cite{Hei,HKST}. 

The setting of a metric measure space supporting a doubling measure allows for a significant amount 
of classical zeroth-order analysis to carry through, see~\cite{CW}. When progressing to first-order analysis, however, 
one needs a replacement for the classical notion of differentiability. We define here the notion of an 
upper gradient for a function, which serves as an analogue of the magnitude of the gradient within the 
context of the fundamental theorem of calculus, see~\cite{HK,KM,S}.

\begin{definition}
We say that a Borel function $g:\Omega\rightarrow[0,\infty]$ is an \emph{upper gradient} of a function
$u:\Omega\to\R$ if 
\begin{equation}\label{eq:UG}
|u(y)-u(x)|\le \int_\gamma\! g\, ds
\end{equation}
whenever $\gamma$ is a non-constant compact rectifiable curve in $\Omega$ joining the points $x,y\in\Omega$. 
We interpret this inequality in the sense that $\int_\gamma g\, ds=\infty$ whenever at least one of $u(x)$ or $u(y)$ is not finite. 
\end{definition}

When a Borel function $g$ satisfies~\eqref{eq:UG} for all rectifiable curves except for a collection $\Gamma$ of curves for 
which there exists a Borel function $0\leq\rho\in L^p(\Omega)$ satisfying $\int_{\gamma}\rho \,ds=\infty$ for all 
$\gamma\in \Gamma$, we say that $g$ is a \emph{$p$-weak upper gradient}. For a function $u$, we denote 
by $D_p(u)$ the collection of all $p$-weak upper gradients of $u$ that are in $L^p(\Omega)$. If it is non-empty, then 
there is an element that is minimal in the sense of $L^p$-norm and pointwise almost everywhere, see for instance~\cite{Ha}
or~\cite[Theorem~6.3.20]{HKST}.
This minimal element is called a \emph{minimal $p$-weak upper gradient}, and plays the role of the norm of gradients of differentiable
functions. 

With the notion of upper gradient, we are able to define what is meant by a Poincaré inequality in the setting of a metric measure space.

\begin{definition}
For $1\leq q,p<\infty$, we say that a metric measure space $(\Omega, d, \mu)$ supports a 
$(q,p)$-\emph{Poincar\'e inequality} if there are constants $C>0$ and $\lambda\geq 1$ such that
\[
\left(\fint_B\!|u-u_B|^q\,d\mu\right)^{1/q}\leq C\,\rad(B)\left(\fint_{\lambda B}\!g^p\,d\mu\right)^{1/p}
\]
for all function-upper gradient pairs $(u,g)$, where $\rad(B)$ is the radius of $B$ and $u_B:=\fint_{B}u\,d\mu$. 
\end{definition}

It follows from H\"older's inequality that if a space supports a $(q,p)$-Poincar\'e inequality for some 
$1\leq q,p<\infty$, then it satisfies a $(q',p')$-Poincar\'e inequality for all $q'\leq q$ and $p'\geq p$. 
Conversely, if $(\Omega, d, \mu)$ supports a $(1,p)$-Poincar\'e inequality and $\mu$ is doubling, 
then it supports a $(p,p)$-Poincar\'e inequality, see for instance~\cite[Theorem~9.1.2]{HKST} or~\cite[Theorem~5.1]{HaKo}.
We will often shorten the notation and say that $\Omega$ supports a $p$-Poincar\'e inequality when $\Omega$ supports a $(1,p)$-Poincar\'e inequality. 

A metric measure 
space supporting a Poincar\'e inequality enjoys strong geometric connectivity properties; 
in particular, the space is quasiconvex, see for instance~\cite{Chee} or~\cite{HKST}
for the case of $X$ being complete in addition to $\mu$ being doubling, and~\cite{DJS} for the case that
$\mu$ is doubling but $X$ is only locally compact.

In the study of boundary-value problems, we are interested in the behavior of functions on $\Omega$ near or on 
its boundary $\partial\Omega$. For this reason, we need to be able to speak about the dimension of the boundary in a 
measure-theoretic way. 
However, at this level of generality of non-smoothness, there is no single natural dimension of a measure to 
consider; hence, 
it makes most sense to consider a notion of 
codimension of the boundary measure with respect to the ambient measure, 
as in the following definition. The condition in this definition is reminiscent of an Ahlfors regularity condition for measures,
but here we relate two different measures to each other. Recall that the measure $\mu$ is supported on the domain $\Om$.

\begin{definition}
For $\Theta>0$, we say that a Borel regular measure $\nu$ on $\partial\Om$ is $\Theta$-\emph{codimensional} to 
$\mu$ if there exists a constant $C>0$ such that
\begin{equation}\label{eq:Ahlfors}
C^{-1} \frac{\mu(B(x,r)\cap\Omega)}{r^\Theta}\leq \nu(B(x,r)\cap\partial\Omega)\leq C \frac{\mu(B(x,r)\cap\Omega)}{r^\Theta}
\end{equation}
for all $x\in\partial\Omega$ and $0<r<2\diam(\partial\Omega)$.
\end{definition}

Balls in $\partial\Om$ are intersection of $\partial\Om$ with balls in $\overline{\Om}$, centered at points in $\partial\Om$.
The measure $\nu$ on $\partial\Om$ can also be considered as a measure on $\overline{\Om}$, but supported on
$\partial\Om$; while the measure $\mu$, which lives on $\Om$, can be extended as a null measure to $\partial\Om$.
Therefore, in this note, we will consider both $\mu$ and $\nu$ as measures on $\overline{\Om}$, but interpreted as above,
and by balls we mean balls in $\overline{\Om}$. Hence,~\eqref{eq:Ahlfors} can also be written as
\[
C^{-1} \frac{\mu(B(x,r))}{r^\Theta}\leq \nu(B(x,r))\leq C \frac{\mu(B(x,r))}{r^\Theta}
\]
for $x\in\partial\Om$.

It follows from the above definition that if $\mu$ is doubling, then so is $\nu$. As the boundary of a bounded domain is 
itself bounded, if $\mu$ is doubling and $\Omega$ is bounded, then $\nu(\partial\Omega)<\infty$. 

\begin{standing}\label{stand:down}
Throughout this note, we fix $1<p<\infty$, set $\Om$ to be a bounded uniform domain 
equipped with a doubling measure $\mu$ such that the metric measure space $(\Om,d,\mu)$ supports 
a $p$-Poincar\'{e} inequality, and assume that the boundary $\partial\Omega$ 
supports a measure $\nu$ that is $\Theta$-codimensional to $\mu$ for some $0<\Theta<p$. 
Note that in this setting, $\nu(\partial\Omega)<\infty$. 
\end{standing}

\subsection{Function spaces} 

In this section, we define the function spaces that will be considered in this paper. We consider 
functions defined on 
$\Omega$ that are in an appropriate Sobolev class. Within the context of metric measure 
spaces, there exist many possible approaches to Sobolev classes; here, we begin with the 
approach by upper gradients, see~\cite{HKST,S}. 

\begin{definition}
We say that a measurable function $u$ is in the \emph{Dirichlet space} $D^{1,p}(\Omega)$ if 
\[
\|u\|^p_{D^{1,p}(\Omega)}:=\inf_{g\in D_p(u)}\int_{\Omega}\!g^p\,d\mu<\infty.
\]
\end{definition}

By the Poincar\'e inequality, every function $u\in D^{1,p}(\Omega)$ is in $L^p_{\loc}(\Omega)$. 
Since $\Omega$ is bounded and the Poincar\'e inequality holds globally, we actually have that 
$D^{1,p}$-functions are in $L^p(\Om)$, and so 
$D^{1,p}(\Om)=N^{1,p}(\Om)$ as sets; see the discussion following Definition~\ref{def:inhomB} for
similar phenomenon related to Besov spaces.
Recall that $D_p(u)$ is the collection of all upper gradients of $u$ that are in $L^p(\Omega)$.

\begin{definition}
We say that a measurable function $u$ is in $\widetilde{N}^{1,p}(\Omega)$ if 
\[
\|u\|_{N^{1,p}(\Omega)}:=\|u\|_{L^p(\Omega)}+\|u\|_{D^{1,p}(\Omega)}<\infty.
\]
This defines a seminorm; taking the quotient space modulo the null-space of the seminorm 
$\|\cdot\|_{N^{1,p}(\Omega)}$ yields the \emph{Newton-Sobolev} space $N^{1,p}(\Omega)$. 
\end{definition}

Functions in $L^p(\Omega)$ are well defined up to sets of $\mu$-measure zero. However,
changing a function from $N^{1,p}(\Omega)$ on arbitrary sets of $\mu$-measure zero can result
in a function that no longer has a $p$-weak upper gradient in $L^p(\Om)$; thus care needs to be taken if
one wishes to modify functions in $N^{1,p}(\Om)$.
It turns out that if $u_1$ and $u_2$ are two 
functions in $\widetilde{N}^{1,p}(\Omega)$, then $\|u_1-u_2\|_{N^{1,p}(\Omega)}=0$ if and only if the set 
$\{x\in\Omega:u_1(x)\neq u_2(x)\}$ has $p$-capacity zero, as defined below. For more on this, we
refer the reader to~\cite{S, HKST}.
	
\begin{definition}
For a set $E\subset\Omega$, we define its \emph{$p$-capacity} as the number
\[
\capa(E):=\inf\|u\|_{N^{1,p}(\Omega)},
\]
where the infimum is taken over all $u\in N^{1,p}(\Omega)$ satisfying $0\leq u \leq 1$ on all of $\Omega$ and $u=1$ on $E$. 
\end{definition}
As $p$ is fixed throughout this note, we will sometimes refer to the capacity of a set as opposed to the $p$-capacity of a set.

Another approach to Sobolev spaces of functions defined on metric measure spaces is due to 
a combination of a differential structure developed by Cheeger for Lipschitz functions, see~\cite{Ch},
and the extension of such a structure to functions in $N^{1,p}$-classes, see~\cite[Theorems~9 and~10]{FHK}.

\begin{definition}
A metric measure space $(\Om,d,\mu)$ is said to support a \emph{Cheeger differential structure} of 
dimension $N\in\N$ if there exists a collection of coordinate patches $\{(\Om_\alpha,\pip_\alpha)\}$ and 
a $\mu$-measurable inner product $\langle \cdot,\cdot\rangle_{x}$, $x\in \Om_\alpha$, on $\R^N$ such that
\begin{enumerate}
	\item each $\Om_\alpha$ is a measurable subset of $\Om$ with positive measure and 
$\bigcup_\alpha \Om_\alpha$ has full measure;
	\item each $\pip_\alpha:\Om_\alpha\rightarrow\R^N$ is Lipschitz;
	\item for every function $u\in D^{1,p}(\Om)$, for $\mu$-a.e. $x\in \Om_\alpha$ there is a 
vector $\nabla u(x)\in\R^N$ such that
\[
\elimsup\limits_{\Om_\alpha\ni y\rightarrow x}\frac{|u(y)-u(x)-\langle\nabla u(x),\pip_\alpha(y)-\pip_\alpha(x)\rangle_{x}|}{d(y,x)}=0.
\]
The vector $\nabla u$ is called the \emph{Cheeger gradient} of $u$.
\end{enumerate}
\end{definition}

Note that there may be more than one possible Cheeger differential structure on a given space, and so 
we choose the one with the further property that $|\nabla u(x)|_x$ is comparable to the 
$p$-weak upper gradient of $u$ for $\mu$-a.e. $x\in\Omega$, where 
$|\nabla u(x)|_x^2=\langle\nabla u(x),\nabla u(x)\rangle_x$, see~\cite{Chee, HKST}. 
This choice of differential structure can 
replace the upper gradient structure used to define $N^{1,p}(\Omega)$ to obtain an equivalent 
definition when $\Omega$ is a doubling metric measure space supporting a Poincar\'e inequality, 
provided that $1<p<\infty$. In what follows, we suppress the dependence of the inner product and norm on the point $x$. 

While Sobolev spaces form the natural space of functions on $\Omega$ to consider, the natural 
function-space on $\partial\Omega$ to consider is the Besov space. 

\begin{definition}
We say that a measurable function $f$ is in $\dot{B}^\theta_{p,p}(\partial\Omega)$ if 
\[
\|f\|^p_{HB^\theta_{p,p}(\partial\Omega)}
:= \int_{\partial\Omega}\int_{\partial\Omega}\!\frac{|f(y)-f(x)|^p}{d(y,x)^{\theta p}\nu(B(y,d(y,x))) }\,d\nu(x)\,d\nu(y)<\infty.
\]
The quantity $\|f\|^p_{HB^\theta_{p,p}(\partial\Omega)}$ is called the \emph{Besov energy} of $f$ in this note.
\end{definition}

For constant functions $c$, we have $\|c\|_{HB^\theta_{p,p}(\partial\Omega)}=0$. In fact, 
the collection of all functions with this seminorm zero is precisely the collection of all $\nu$-a.e.~constant functions. 

\begin{definition}
We define the \emph{homogeneous Besov space} $HB^\theta_{p,p}(\partial\Omega)$ 
to be  $\dot{B}^\theta_{p,p}(\partial\Omega)/\R$. The natural norm on $HB^\theta_{p,p}(\partial\Omega)$ 
coming from being a quotient space modulo constants is given by
\[
\inf_{c\in\R}\|f+c\|_{HB^\theta_{p,p}(\partial\Omega)}
\]
whenever $f\in \dotb(\partial\Om)$.
However, $\|f+c\|_{HB^\theta_{p,p}(\partial\Omega)}$ trivially equals $\|f\|_{HB^\theta_{p,p}(\partial\Omega)}$ 
for every constant $c$, and so
\[
\inf_{c\in\R}\|f+c\|_{HB^\theta_{p,p}(\partial\Omega)}=\|f\|_{HB^\theta_{p,p}(\partial\Omega)}.
\]
\end{definition}

\begin{definition}\label{def:inhomB}
We set the \emph{inhomogeneous Besov space} $B^\theta_{p,p}(\partial\Omega)$ to be the space
$\dot{B}^\theta_{p,p}(\partial\Omega)\cap L^p(\partial\Omega)$. 
The natural norm on $B^\theta_{p,p}(\partial\Omega)$, coming from being an intersection of normed spaces, is given by 
\[
\max\{\|f\|_{HB^\theta_{p,p}(\partial\Omega)}, \|f\|_{L^p(\partial\Omega)} \}\approx \|f\|_{HB^\theta_{p,p}(\partial\Omega)}+ \|f\|_{L^p(\partial\Omega)}.
\]
\end{definition}

It follows from \cite[Lemma 2.2]{CKKSS} that every function in the class
$\dot{B}^\theta_{p,p}(\partial\Omega)$ is necessarily in $L^p(\partial\Om)$ since $\nu(\partial\Om)$ is finite. Hence,
\emph{as sets}, we have also that $\inb=\dotb(\partial\Om)$.

In the next section and in the rest of the paper, we exploit connections between Sobolev functions on $\Om$ and
Besov functions on $\partial\Om$. We need corresponding connections between $\homb$ and a homogeneous space associated
with Sobolev classes. In our setting, the \emph{homogeneous Newton-Sobolev class} $\homn$ is the natural partner.
To define $\homn$, we consider an equivalence relationship $\sim$ on $\inn$ as follows: given $f_1,f_2\in\inn$, we say
that $f_1\sim f_2$ if 
\[
\inf_{g\in D_p(f_1-f_2)}\int_\Om\! g^p\, d\mu=0.
\]
Then $\homn$ is the collection of all equivalence classes from $\inn$; note that if $f\in\homn$ and $f_1,f_2$ are
two representatives of $f$, then by the Poincar\'e inequality we have that $f_1-f_2$ is constant. Thus, equivalence classes
consist of collections of Sobolev functions (that is, functions from $\inn$) that differ from each other only by a constant.
The norm $\|f\|_{\homn}$ on $\homn$ is given by
\[
\|f\|_{\homn}=\|f_1\|_{D^{1,p}(\Om)},
\]
for any choice of representative from the class $f$. Note that this norm is independent of the choice of representative $f_1$ of $f$.

\subsection{Trace and extension operators}
Here, we present the known trace and extension results in the setting of bounded uniform domains due to Mal\'y~\cite{Mal}. 
Let $\theta=1-\Theta/p$. As $0<\Theta<p$, we have that $0<\theta<1$. 

\begin{theorem}\label{th:extension}
Under the standing assumptions~\ref{stand:down} above, 
there is a surjective bounded linear trace operator $\Tr: N^{1,p}(\Om)\to \dotb(\partial\Om)$ 
with 
\[
\|\Tr(u)\|_{HB^{\theta}_{p,p}(\partial\Omega)}\leq \|\!\Tr\!\|\|\nabla u \|_{L^p(\Omega)}
\]
and there is a
bounded linear extension operator $E:\dotb(\partial\Omega)\to N^{1,p}(\Om)$ with 
\[
\|\nabla Ef\|_{L^p(\Omega)}\leq \|E\|\|f\|_{HB^\theta_{p,p}(\partial\Omega)}.
\] 
Moreover, $\Tr\circ E$ is the identity operator on $\dot{B}^{\theta}_{p,p}(\partial\Omega)$. 
The trace $\Tr(u)$ of a function $u$ is given as follows. For each $\xi\in\partial\Om$,
$\Tr(u)(\xi)$ is the number (if it exists) satisfying
\[
\lim_{r\to 0^+}\fint_{B(\xi,r)}|u-\Tr(u)(\xi)|\, d\mu=0.
\]
\end{theorem}

\begin{remark}\label{rem:pass2quotient-trace}
Since we are more concerned with $\homb$ rather than $\dotb(\partial\Om)$, the
operator $\Tr$ in the rest of this note is actually the composition of the trace operator from Theorem~\ref{th:extension}
above with the canonical quotient map from $\dotb(\partial\Om)$ to $\homb$. 
In particular, 
\[
\Tr:\homn\to\homb
\] 
is a bounded linear trace operator.
A similar argument also gives us
\[
E:\homb\to\homn
\] 
as a bounded linear extension operator such that $\Tr\circ E$ is the identity map on $\homn$.
\end{remark}

The original theorems of Mal\'y~\cite{Mal} are expressed in terms of the inhomogeneous Besov space. However,
when $\nu(\partial\Omega)<\infty$, as we assume here, the sets $\dot{B}^{\theta}_{p,p}(\partial\Omega)$ and 
$B^{\theta}_{p,p}(\partial\Omega)$ are the same, and so this theorem can be interpreted in terms of the 
homogeneous Besov space since only the Besov energy is involved. 
Although~\cite{Mal} gives estimates for the Besov energy of traces in terms of upper gradient energy
of the Sobolev functions, we can replace the upper gradient energy in~\cite{Mal} with the Cheeger energy 
as the two are comparable.

\section{Boundary-value problems}

The focus of this section is to establish the existence of solutions to the Neumann boundary-value problems
where the Neumann data is given by an operator. Previously, existence of solutions to the Neumann boundary-value
problems was established in~\cite{MS} was established under the condition that the Neumann data is given by
a function on $\partial\Om$, and so our result in this section extends the discussion from~\cite{MS} to a wider
class of Neumann data.

We first discuss the necessary background regarding the Dirichlet boundary-value problem 
when the boundary data is in an appropriate Besov class. We then formulate the Neumann 
boundary-value problem for boundary data in the dual space of an appropriate Besov class
and end this section by demonstrating the existence of solutions to this problem. 

\begin{definition}\label{def:Dirich-inhom}
For $f\in B^\theta_{p,p}(\partial\Omega)$, we say that a function $u_f\in N^{1,p}(\Omega)$ is a \emph{solution to the Dirichlet problem} with boundary data $f$ if
\[
\int_{\Omega}\!|\nabla u_f|^p\,d\mu\leq \int_{\Omega}\!|\nabla (u_f+v)|^p\,d\mu
\]
for all $v\in N^{1,p}_0(\Omega)$ and $\Tr(u_f)=f$ pointwise $\nu$--a.e.~in $\partial\Omega$. 
From~\cite{KinnSh}, we know that we can modify such a solution on a set of $\mu$-measure zero
so that it is locally H\"older continuous on $\Om$; henceforth, it is this H\"older continuous representative that we denote
$u_f$.
\end{definition}

In other words, $u_f$ is a minimizer for the functional
\[
\frac1p\int_{\Omega}\!|\nabla u|^p\,d\mu
\]
over the class of $u\in N^{1,p}(\Omega)$ satisfying $\Tr(u)=f$ pointwise $\nu$--a.e.~in $\partial\Omega$. 
This variational perspective provides the following characterization for a function $u$ to be \emph{$p$-harmonic} in $\Om$:
\begin{equation}\label{eq:EL4inhom}
\int_{\Omega}\!|\nabla u|^{p-2} \,\langle \nabla u, \nabla v \rangle\,d\mu=0 
\end{equation}
for every $v\in N^{1,p}_0(\Omega)$. 
Note that for each $f\in\inb$, the solution $u_f$ to the Dirichlet problem with boundary data $f$ exists, see 
for instance~\cite{Sh2}.

\begin{definition}\label{def:Dir-Hom}
When $f\in\homb$, by a \emph{solution to the Dirichlet problem} with boundary data $f$ we mean the \emph{collection}
of solutions to the Dirichlet problem with boundary data $\tilde{f}$ for each representative function 
$\tilde{f}$ of $f$; note that $f$ is itself
a collection of functions, two of which differ only by a constant $\nu$-almost everywhere.
We denote this collection also by the notation $u_f$.
\end{definition}

\begin{remark}\label{rem:zero-dirich}
Note that if $f_1,f_2$ are both representatives of $f\in\homb$, then $f_2=f_1+c$ for some constant $c\in\R$;
in this case, $u_{f_2}=u_{f_1}+c$ \emph{everywhere} on $\Om$, where $u_{f_1}$ and $u_{f_2}$ are
the solutions from Definition~\ref{def:Dirich-inhom}.
So the collection $u_f$ that forms a solution to the Dirichlet boundary value problem with boundary data $f\in\homb$
is itself made up of functions of the form $u_{f_1}+c$, $c\in\R$, and hence $u_f\in\homn$. 

Note that if we replace $v\in N^{1,p}_0(\Om)$ with $v+c$ in~\eqref{eq:EL4inhom} whenever $c\in\R$, the
equation~\eqref{eq:EL4inhom} continues to remain valid; that is, that equation requires only that $v$ has constant
trace on $\partial\Om$.
Therefore the analog of the 
Euler-Lagrange equation~\eqref{eq:EL4inhom} in this homogeneous setting is that whenever $u$ is a representative of
$u_f$ and $v$ is a representative of $v_0\in HN_0^{1,p}(\Om)$, we have
\begin{equation}\label{eq:EL4hom}
\int_{\Omega}\!|\nabla u|^{p-2} \,\langle \nabla u, \nabla v \rangle\,d\mu=0.
\end{equation}
Here, the class $HN_0^{1,p}(\Om)$ is the collection of all $v_0\in\homn$ for which each representative $v$ of $v_0$
has constant trace at $\partial\Om$ in the sense of Theorem~\ref{th:extension}. Note that if $v,w$ are two representatives
of $v_0\in\homn$, then $v-w$ is constant, and so $\nabla v=\nabla w$ $\mu$-a.e.~in $\Om$. Thus, we can denote
by $\nabla v_0$ the Cheeger derivative $\nabla v$ of any (equivalently, every) representative $v$ of $v_0$.
\end{remark}

\begin{remark}
Given that the exposition of this note is about the homogeneous classes, we emphasize the distinction between
elements of the inhomogeneous spaces and their corresponding homogeneous spaces. Two functions in 
the space $\inn$ are the same function if, as functions on the metric space $\Om$, they differ only on a subset
of $\Om$ of capacity zero - which is a more stringent requirement than differing only on a set of $\mu$-measure zero.
Two functions in $\homn$ are the same if representatives of these functions, as pointwise defined functions on $\Om$,
differ by only a constant outside a set of zero capacity.
\end{remark}

In this setting, it is known that a $p$-harmonic function $u_f\in N^{1,p}(\Omega)$ exists for any choice of 
$f\in B^\theta_{p,p}(\partial\Omega)$. Solutions to the Neumann problem in this setting are known for boundary 
data in an appropriate integrable function class. In the present note, we extend this to include boundary data in the 
form of an arbitrary bounded linear functional $L$ on $HB^\theta_{p,p}(\partial\Omega)$.

\begin{definition}
Given an operator $L\in \left(HB^\theta_{p,p}(\partial\Omega)\right)^*$, we say that a function $u_L\in HN^{1,p}(\Omega)$ is a 
\emph{solution to the Neumann problem} with boundary data $L$ if 
\[
\frac1p\int_{\Omega}\!|\nabla u_L|^p\,d\mu - L(\Tr(u_L))\leq\frac1p\int_{\Omega}\!|\nabla v|^p\,d\mu - L(\Tr(v)) 
\]
for all $v\in HN^{1,p}(\Omega)$. 
\end{definition}

This is equivalent to saying that $u_L$ is a minimizer for the functional
\[
I(u):=\frac1p\int_{\Omega}\!|\nabla u|^p\,d\mu - L(\Tr(u))
\]
over the class of $u\in HN^{1,p}(\Omega)$. The following lemma gives us a characterization of solutions to the 
Neumann problem in the form of the Euler-Lagrange equation associated with the functional $I$. 

\begin{lemma}\label{lem:EL}
Let $L\in\left(HB^\theta_{p,p}(\partial\Omega)\right)^*$. Then $u$ is a minimizer of $I$ if and only if
\begin{equation}\label{eq:EL}
\int_{\Omega}\!|\nabla u|^{p-2}\langle \nabla u, \nabla v \rangle\,d\mu - L(\Tr(v))=0
\end{equation}
for every $v\in  HN^{1,p}(\Omega)$.  
\end{lemma}

Note that $L\in\left(HB^\theta_{p,p}(\partial\Omega)\right)^*$ necessarily maps constant functions to $0$. 
As such, this lemma implies, taking $v\in HN^{1,p}_0(\Omega)\subset HN^{1,p}(\Omega)$, 
that a solution to the Neumann problem is necessarily $p$-harmonic; here $HN^{1,p}_0(\Om)$ is 
as described in Remark~\ref{rem:zero-dirich}.

\begin{proof}
For $v\in HN^{1,p}(\Omega)$ and $h\in\R$, consider 
\[
I(u+hv) =\frac1p\, \int_{\Omega}\!|\nabla (u+hv)|^p\,d\mu - L(\Tr(u+hv)).
\]
If $u$ is a minimizer of $I$, then the map $h\mapsto I(u+hv)$ has a minimum at $h=0$, whence $\left.\frac{d}{dh}I(u+hv)\right|_{h=0}=0$. By linearity,
\[
L(\Tr(u+(h+\eta)v))-L(\Tr(u+h v))= \eta L(\Tr(v))
\]
for any $\eta\in\R$, and so 
\[
\frac{d}{dh}L(\Tr(u+h v))=L(\Tr(v)).
\]
We also calculate
\[
\frac{d}{dh}\int_{\Omega}\!|\nabla (u+hv)|^p\,d\mu= p\int_{\Omega}\!|\nabla (u+h v)|^{p-2}\langle \nabla(u+h v), \nabla v\rangle\,d\mu,
\]
and so, setting $h=0$, we obtain
\[
\left.\frac{d}{dh}I(u+hv)\right|_{h=0}=\int_{\Omega}\!|\nabla u|^{p-2}\langle \nabla u, \nabla v\rangle\,d\mu-L(\Tr(v)). 
\]

Conversely, assume that $u$ satisfies \eqref{eq:EL} for every $v\in  HN^{1,p}(\Omega)$. 
The energy $I$ is convex on $\homb$ Ssnce $1<p<\infty$; therefore, a function $u$ minimizes $I$ if and only if
for each $v\in\homb$ we have that
\[
\frac{d}{dh}I(u+hv)\bigg\vert_{h=0}=0,
\]
which in turn is equivalent to the validity of~\eqref{eq:EL}, completing the proof. 
\end{proof}

\begin{lemma}\label{lem:1}
	Given $L\in\left(HB^\theta_{p,p}(\partial\Omega)\right)^*$ and the functional $I$ on $HN^{1,p}(\Omega)$ given by 
	\[
	I(v):=\frac1p\,\int_{\Omega}\!|\nabla v|^p\,d\mu - L(\Tr(v)),
	\]
	the quantity $\alpha:=\inf\{I(v):v\in HN^{1,p}(\Omega)\}$ is finite. 
\end{lemma}

Recall that $L\in\left(HB^\theta_{p,p}(\partial\Omega)\right)^*$ necessarily maps constant functions to 0.

\begin{proof}
	The class of constant functions, $v$, is in $HN^{1,p}(\Omega)$ and $\Tr(v)$ is the class of constant functions on
$\partial\Om$, so $L(v)=0$. Hence, $I(v)=0$, 
and so $\alpha\leq 0$. Thus, it suffices to show that the set $\{I(v):v\in HN^{1,p}(\Omega)\}$ is bounded below.
	
Let $v\in HN^{1,p}(\Omega)$. Then, using the boundedness of $L$ and $\Tr$, as well as Theorem~\ref{th:extension}, we have that 
\begin{align}
	I(v)&\geq \frac1p\,\|\nabla v\|^p_{L^p(\Omega)} - \|L\|\|\Tr(v)\|_{HB^\theta_{p,p}(\Omega)}\notag\\
		&\geq \frac1p\,\|\nabla v\|^p_{L^p(\Omega)} - \|L\|\|\!\Tr\!\|\|\nabla v\|_{L^p(\Omega)}. \label{eq:PT}
\end{align}
	
	For $C>0$, the map $x\mapsto \frac{1}{p}x^p-Cx$ for $x>0$ is bounded below with a minimum of 
	$\frac{1-p}{p}C^{\frac{p}{p-1}}$. As such, it follows that 
	\begin{equation}\label{eq:Iv-est}
		I(v)\ge \frac{1-p}{p}\left(\|L\|\|\!\Tr\!\|\right)^{\frac{p}{p-1}}.
	\end{equation}
	This lower bound is independent of $v$, and so $\alpha$ is bounded below. 
\end{proof}

\begin{theorem}\label{th:2}
	Given $L\in\left(HB^\theta_{p,p}(\partial\Omega)\right)^*$, there exists a solution $u_L$  on $\Omega$ to the 
Neumann problem with boundary data $L$. 
\end{theorem}

If $L$ is the zero functional, then $I$ is minimized by $u_L=0$. 
In what follows, we assume that $\|L\|>0$. 

\begin{proof}
	Recall the definition of $\alpha$ from Lemma~\ref{lem:1} above. Assuming that $||L||>0$, we know that $\alpha\ne 0$.
	
	Let $\{u_k\}\subset HN^{1,p}(\Omega)$ be a sequence such that $I(u_k)\rightarrow \alpha$ as 
	$k\rightarrow\infty$ and $I(u_k)\leq 2|\alpha|$ for all $k$. 
	Since $u_k\in\homn$, for each $k$ we choose a representative of $u_k$, also denoted by $u_k\in \inn$, such that 
	$(u_k)_{\Omega}=\frac{1}{\mu(\Omega)}\int_{\Omega}u_k\,d\mu=0$. From this and the fact that 
	$\Omega$ is bounded, by the $(p,p)-$Poincar\'e inequality 
	applied to the ball that is $\Om$,
	we have that
	\begin{equation}\label{eq:ppP}
		\int_{\Omega}\!|u_k|^p\,d\mu = \int_{\Omega}\!|u_k-(u_k)_{\Omega}|^p\,d\mu \leq C\int_{\Omega}\!|\nabla u_k |^p\,d\mu
	\end{equation} 
	where $C$ is a constant depending on $\Omega$, but not $u_k$. 
	
	From~\eqref{eq:PT}, it follows that for each $k$, 
	\[
	I(u_k)\geq \|\nabla u_k\|_{L^p(\Omega)}\left[\frac1p\,\|\nabla u_k\|^{p-1}_{L^p(\Omega)} - \|L\|\|\!\Tr\!\|\right]. 
	\] 
	From this it follows that either $\|\nabla u_k\|^{p-1}_{L^p(\Omega)}\leq 2 p\|L\|\|\!\Tr\!\|$ or else 
	\[
	2|\alpha|\geq I(u_k)\geq \|\nabla u_k\|_{L^p(\Omega)}\big[2\,\|L\|\|\!\Tr\!\| - \,\|L\|\|\!\Tr\!\|\big]=\|L\|\|\!\Tr\!\|\|\nabla u_k\|_{L^p(\Omega)}.
	\]
	In either case, we have that the sequence $\{\|\nabla u_k\|_{L^p(\Omega)}\}$ is bounded with 
	\begin{equation}\label{eq:grad-u2Tr}
		\|\nabla u_k\|_{L^p(\Omega)}\leq (2p\,\|L\|\|\!\Tr\!\|)^{\frac{1}{p-1}}+\frac{2|\alpha|}{\|L\|\|\!\Tr\!\|}
	\end{equation}
	for each $k$. Hence, \eqref{eq:ppP} implies that the sequence $\{u_k\}$ is bounded in $N^{1,p}(\Omega)$. 
	Since this is a reflexive space, there is a function $u\in N^{1,p}(\Omega)$ to which a subsequence of $\{u_k\}$, 
	which we will continue to denote by $\{u_k\}$, converges weakly. We will now show that $u$ is a minimizer for $I$, 
	that $u$ is $p$-harmonic, and that $L=L_{\Tr(u)}$. Here again, we denote the equivalence class in $\homn$ corresponding
	to $u\in \inn$ also by $u$.
	
	Mazur's lemma implies that there is a sequence of finite convex combinations 
	$\{\widetilde{u}_n\}\subset\inn$ given by
	\[
	\widetilde{u}_n=\sum_{k=n}^{N(n)}\lambda_{n,k}u_k
	\]  
	for  
	$n=1,2,\ldots$, such that $\{\widetilde{u}_n\}$ converges strongly 
	to $u$ in $N^{1,p}(\Omega)$. 
	Continuing our abuse of notation, we also denote the equivalence classes in $\homn$ with representative 
	$\widetilde{u}_n\in \inn$ also by $\widetilde{u}_n$. Then the strong convergence of $\widetilde{u}_n$ to $u$ in
	the norm of $\inn$ implies the strong convergence of the sequence of classes $\widetilde{u}_n$ to the equivalence
	class $u$ in the norm of $\homn$.
	From this strong convergence and the boundedness of the linear
	operators $L$ and $\Tr$ on $\homn$, it 
	follows that $L(\Tr(\widetilde{u}_n))\rightarrow L(\Tr(u))$ as $n\rightarrow\infty$; moreover,
	$\|\nabla \widetilde{u}_n\|_{L^p(\Omega)}\rightarrow\|\nabla u\|_{L^p(\Omega)}$ as $n\rightarrow\infty$ since 
	\[
	\left|\|\nabla \widetilde{u}_n\|_{L^p(\Omega)}-\|\nabla u\|_{L^p(\Omega)}\right|
	\leq\|\nabla \widetilde{u}_n-\nabla u\|_{L^p(\Omega)}\leq \|\widetilde{u}_n-u\|_{N^{1,p}(\Omega)}.
	\]
	Thus,
	\begin{align}
		I(\widetilde{u}_n)
		&=\frac1p\,\int_{\Omega}\!|\nabla \widetilde{u}_n|^p\,d\mu - L(\Tr(\widetilde{u}_n))\notag\\
		&\rightarrow \frac1p\,\int_{\Omega}\!|\nabla u|^p\,d\mu - L(\Tr(u))=I(u).\label{eq:Iun2Iu}
	\end{align}
	We aim to show that $u$ is a minimizer of $I$ by showing that $I(\widetilde{u}_n)\rightarrow\alpha$ as $n\rightarrow\infty$. 
	To that end, notice that for each $n$,
	\begin{align*}
		(pI(\widetilde{u}_n)+pL(\Tr(\widetilde{u}_n)))^{1/p} & = \|\nabla \widetilde{u}_n\|_{L^p(\Omega)} 
		\leq \sum_{k=n}^{N(n)}\lambda_{n,k}\|\nabla u_k\|_{L^p(\Omega)} \\
		& \leq \left(\sum_{k=n}^{N(n)}\lambda_{n,k}\|\nabla u_k\|^p_{L^p(\Omega)}\right)^{\frac{1}{p}}\left(\sum_{k=n}^{N(n)}\lambda_{n,k}\right)^{\frac{p}{p-1}}\\
		& = \left(\sum_{k=n}^{N(n)}\lambda_{n,k}\|\nabla u_k\|^p_{L^p(\Omega)}\right)^{\frac{1}{p}}
	\end{align*}
	by H\"older's inequality. Similarly, $\|\nabla u_k\|_{L^p(\Omega)}=(pI(u_k)+p\,L(\Tr(u_k)))^{1/p}$ for each 
	$k$, and so from the linearity of $L$, we have that
	\begin{align*}
		(I(\widetilde{u}_n)+L(\Tr(\widetilde{u}_n)))^{1/p} &\leq \left(\sum_{k=n}^{N(n)}\lambda_{n,k}(I(u_k)+L(\Tr(u_k)))\right)^{\frac{1}{p}}\\
		&= \left(\sum_{k=n}^{N(n)}\lambda_{n,k}I(u_k) + L(\Tr(\widetilde{u}_n)) \right)^{\frac{1}{p}}.
	\end{align*}
	At this point, we use that $\{u_k\}$ was chosen to be a minimizing sequence. Fix $\eps>0$ and select $N$ such that 
	$\alpha\leq I(u_k)< \alpha + \eps$ for all $k\geq N$. Then, for $n\geq N$, we have that 
	\begin{align*}
		(I(\widetilde{u}_n)+L(\Tr(\widetilde{u}_n)))^{1/p}
		&<\left(\sum_{k=n}^{N(n)}\lambda_{n,k}(\alpha+\eps) + L(\Tr(\widetilde{u}_n)) \right)^{\frac{1}{p}}\\
		&=\left(\alpha+\eps + L(\Tr(\widetilde{u}_n)) \right)^{\frac{1}{p}},
	\end{align*}
	and so $\alpha \leq I(\widetilde{u}_n)<\alpha+\eps$ for such $n$. That is, $I(\widetilde{u}_n)\rightarrow\alpha$ as $n\rightarrow\infty$. 
	From~\eqref{eq:Iun2Iu}, we now know that $I(u)=\alpha$, that is, $u$ is a minimizer of the energy functional $I$. Hence
	we can set $u_L=u$.
\end{proof}

\section{Betwixt Dirichlet and Neumann}

In this section, we construct the Dirichlet-to-Neumann and Neumann-to-Dirichlet maps, and show that they are bounded 
in terms of the norms of the extension and trace operators, respectively, whose existence is guaranteed by 
Theorem~\ref{th:extension}. 

We begin this section by showing the existence of the Dirichlet-to-Neumann map in this setting. 

\begin{theorem}\label{th:f-to-Lf}
Given $f\in \homB(\partial\Omega)$ with solution $u_f$ to the Dirichlet problem with boundary data $f$, 
there exists a bounded linear functional $L_f$ in $\left(HB^\theta_{p,p}(\partial\Omega)\right)^*$ given by 
\[
L_f(g)=\int_{\Omega}\!|\nabla u_f|^{p-2} \,\langle \nabla u_f, \nabla Eg \rangle\,d\mu 
\]
such that $u_f$ is a solution to the Neumann problem with boundary data $L_f$. 
\end{theorem}

Note that linearity of $L_f$ follows from the linearity of the extension operator, Cheeger gradient, inner product, and integral. 

\begin{proof}
We begin by verifying that the functional $L_f$ is well defined: that it is independent of the choice of extension 
operator $E$ and of representative of $g$. To this end, suppose that $v_1,v_2\in N^{1,p}(\Omega)$ such that 
$\Tr(v_1)=\Tr(v_2)+c$ $\nu$-almost everywhere for some constant $c\in\R$, whence $v_1-v_2\in HN^{1,p}_0(\Omega)$. Then,
\[
\int_{\Omega}\!|\nabla u_f|^{p-2} \,\langle \nabla u_f, \nabla (v_1-v_2) \rangle\,d\mu =0
\]
since $u_f$ is $p$-harmonic in $\Om$, see for instance Remark~\ref{rem:zero-dirich}. 
Suppose that $g_1,g_2$ are two representatives of some $g\in HB^\theta_{p,p}(\partial\Omega)$. Then $g_1-g_2$ is a 
constant up to a set of $\nu$ measure zero, and so $E(g_1-g_2)$ is a constant up to a set of $\mu$ measure zero. Then, 
\[
\int_{\Omega}\!|\nabla u_f|^{p-2} \,\langle \nabla u_f, \nabla E(g_1-g_2) \rangle\,d\mu =0,
\]
from where it follows that the definition of $L_f(g)$ is independent of the choice of representative of $g$, and of
the extension $Eg_1$ of the representative function $g_1$ of $g$.

Now we show that $L_f$ is indeed bounded. For $g\in HB^\theta_{p,p}(\partial\Omega)$, we have that
\begin{align*}
|L_f(g)| &\leq \int_{\Omega}\! |\nabla u_f|^{p-1} |\nabla Eg| \,d\mu 
\leq \left( \int_{\Omega}\! |\nabla u_f|^p \,d\mu \right)^{\frac{p-1}{p}} \left( \int_{\Omega}\!|\nabla Eg|^p\,d\mu \right)^{\frac{1}{p}} \\ 
&\leq \|E\| \|\nabla u_f\|^{p-1}_{L^p(\Omega)}\|g\|_{HB^\theta_{p,p}(\partial\Omega)}.
\end{align*}
We used the Cauchy-Schwarz inequality and H\"older's inequality, followed by 
the extension
Theorem~\ref{th:extension}. From this calculation it follows that the operator norm of $L_f$ satisfies 
$\|L_f\|\leq  \|E\| \|\nabla u_f\|^{p-1}_{L^p(\Omega)}$. Since $u_f$ is $p$-harmonic, it minimizes the Dirichlet $p$-energy, and so, 
\[
\|\nabla u_f\|_{L^p(\Omega)}\leq \|\nabla Ef\|_{L^p(\Omega)}\leq \|E\|\|f\|_{HB^\theta_{p,p}(\partial\Omega)}.
\] 
It follows that 
\begin{equation}\label{eq:Lf-to-E}
\|L_f\|\leq \|E\|^p \|f\|^{p-1}_{HB^\theta_{p,p}(\partial\Omega)}.
\end{equation}

For $v\in N^{1,p}(\Omega)$, its trace $\Tr(v)$ is in $HB^{\theta}_{p,p}(\partial\Omega)$ and so 
\[
L_f(\Tr(v))=\int_{\Omega}\!|\nabla u_f|^{p-2} \,\langle \nabla u_f, \nabla E\Tr(v) \rangle\,d\mu. 
\]
Since $\Tr\circ E$ is the identity, we know that $\Tr(E\Tr(v))=\Tr(v)$, and therefore
 $E\Tr(v)-v\in N^{1,p}_0(\Omega)$. Since $u_f$ is $p$-harmonic, it follows that 
\[
\int_{\Omega}\!|\nabla u_f|^{p-2} \,\langle \nabla u_f, \nabla (E\Tr(v)-v) \rangle\,d\mu=0,
\]
and so
\[
L_f(\Tr(v))=\int_{\Omega}\!|\nabla u_f|^{p-2} \,\langle \nabla u_f, \nabla v \rangle\,d\mu. 
\]
This is, $u_f$ is a solution to the Neumann problem with boundary data $L_f$.
\end{proof}

\begin{corollary}\label{cor:DtN-norm}
The Dirichlet-to-Neumann map 
\[
\DtN:HB_{p,p}^\theta(\partial\Omega)\rightarrow \left(HB_{p,p}^\theta(\partial\Omega)\right)^*
\] 
given by 
$\DtN(f)=L_f$
is well-defined and bounded, with
\[
\|\DtN\|:=\sup_{\|f\|_{HB_{p,p}^\theta(\partial\Omega)}=1}\|L_f\| \leq \|E\|^p
\]
\end{corollary}

\begin{proof}
From Theorem~\ref{th:f-to-Lf}, it follows that $\DtN$ is well defined and takes Dirichlet data to Neumann data with 
$u_f$ being also the solution to the Neumann problem with Neumann data $\DtN(f)$.
Furthermore, from \eqref{eq:Lf-to-E} in the proof of Theorem~\ref{th:f-to-Lf} it follows that $\DtN$ is bounded with
\[
\sup_{\|f\|_{HB_{p,p}^\theta(\partial\Omega)}=1}\|L_f\| \leq \|E\|^p. 
\]
\end{proof}

We now proceed by constructing the Neumann-to-Dirichlet map. 

\begin{theorem}\label{prop:L-to-TruL}
Given $L\in\left(HB^\theta_{p,p}(\partial\Omega)\right)^*$, let $u_L$ be the solution 
to the Neumann problem with boundary data $L$. Then $u_L$ is the solution to the Dirichlet problem 
with boundary data $\Tr(u_L)$. Moreover, we have $L=L_{\Tr(u_L)}$.
\end{theorem}

\begin{proof}
As $u_L$ is the solution to a Neumann problem, it is $p$-harmonic, and its trace is equal to $\Tr(u_L)$, and 
so $u_L$ is the solution to the Dirichlet problem with boundary data $\Tr(u_L)$. 

For $v\in HN^{1,p}(\Omega)$, its trace $\Tr(v)$ is in $HB^{\theta}_{p,p}(\partial\Omega)$ and so 
by Lemma~\ref{lem:EL},
\[
L(\Tr(v))=\int_{\Omega}\!|\nabla u_L|^{p-2} \,\langle \nabla u_L, \nabla v \rangle\,d\mu.
\]
Since $\Tr\circ E$ is the identity, we know that $\Tr(E\circ\Tr(v))=\Tr(v)$, and therefore
$E\Tr(v)-v\in HN^{1,p}_0(\Omega)$. Since $u_L$ is $p$-harmonic, it follows that 
\[
\int_{\Omega}\!|\nabla u_L|^{p-2} \,\langle \nabla u_L, \nabla (E\circ\Tr(v)-v) \rangle\,d\mu=0,
\]
and so
\begin{align*}
L(\Tr(v))&=\int_{\Omega}\!|\nabla u_L|^{p-2} \,\langle \nabla u_L, \nabla v \rangle\,d\mu \\
&=  \int_{\Omega}\!|\nabla u_L|^{p-2} \,\langle \nabla u_L, \nabla E(\Tr(v)) \rangle\,d\mu = L_{\Tr(u_L)}(\Tr(v)). 
\end{align*}
The result follows from the fact that the trace operator is surjective. 
\end{proof}

\begin{corollary}\label{cor:NtD-norm}
The Neumann-to-Dirichlet map 
\[
\NtD:\left(HB_{p,p}^\theta(\partial\Omega)\right)^*\rightarrow HB_{p,p}^\theta(\partial\Omega)
\] 
given by $\NtD(L):=\Tr(u_L)$
is well-defined and bounded, with 
\[
\|\NtD\|:=\sup_{\|L\|=1}\|\Tr(u_L)\|_{HB_{p,p}^\theta(\partial\Omega)}\le \left[(2p)^{\frac{1}{p-1}}+\frac{2(p-1)}{p}\right]||\!\Tr\!||^{\frac{p}{p-1}}.
\]
\end{corollary}

\begin{proof}
From Theorem~\ref{prop:L-to-TruL}, it follows that $\NtD$ is well defined and takes Neumann data to Dirichlet data with solution being equal to the solution of the Neumann problem. Furthermore, 
by~\eqref{eq:Iv-est} and then taking the infimum over all $v\in HN^{1,p}(\Om)$, we know that
\[
0\ge \alpha\ge \frac{1-p}{p}\, \left(||L||\, ||\!\Tr\!||\right)^{\frac{p}{p-1}}.
\]
So, by~\eqref{eq:grad-u2Tr}, when $L\in(\homB(\partial\Om))^*$ and 
$u_L$ is the corresponding minimizer constructed in Theorem~\ref{th:2}, we have
\begin{align}
	||\nabla u_L||_{L^p(\Om)}&\le \left(2p\,||L||\, ||\Tr||\right)^{\frac{1}{p-1}}+\frac{p-1}{p}\, \frac{2}{||L||\, ||\Tr||}\, \left(||L||\, ||\Tr||\right)^{\frac{p}{p-1}}\notag\\
	&=\left[(2p)^{\frac{1}{p-1}}+\frac{2(p-1)}{p}\right]\, \left(||L||\, ||\Tr||\right)^{\frac{1}{p-1}}. \label{eq:uL-est}
\end{align}
Technically, the inequality~\eqref{eq:grad-u2Tr} is known to hold for the minimizing sequence $\nabla u_k$, but as
this inequality is respected by convex combinations of $\nabla u_k$, and the corresponding convex combination converges
strongly to $\nabla u_L$, this inequality holds also for $\nabla u_L$.

The above inequality, along with Theorem~\ref{th:extension}, implies that $\NtD$ is bounded with
\begin{align*}
\sup_{\|L\|=1}\|\Tr(u_L)\|_{HB_{p,p}^\theta(\partial\Omega)} &\leq \sup_{\|L\|=1}\|\!\Tr\!\|\|\nabla u_L\|_{L^p(\Omega)}\\
&\leq  \left[(2p)^{\frac{1}{p-1}}+\frac{2(p-1)}{p}\right]||\!\Tr\!||^{\frac{p}{p-1}}.
\end{align*}
\end{proof}

\begin{theorem}\label{th:inverse}
The operators $\dtn$ and $\ntd$ are inverses of each other. Moreover,
\[
c_p\ \|\Tr\|^{\tfrac{p}{1-p}}\le \|DtN\|\le \|E\|^p,
\]
where
\[
c_p:=\left[(2p)^{\tfrac{1}{p-1}}+\frac{2(p-1)}{p}\right]^{-1}.
\]
\end{theorem}

\begin{proof}
For every $L\in\left(HB^\theta_{p,p}(\partial\Omega)\right)^*$, we have 
\[
\DtN\circ\NtD(L)=\DtN(\Tr(u_L))=L_{\Tr(u_L)}=L
\]
by Theorem~\ref{prop:L-to-TruL}. 

For every $f\in HB^\theta_{p,p}(\partial\Omega)$, we have 
\[
\NtD\circ\DtN(f)=\NtD(L_f)=\Tr(u_{L_f}),
\]
where $u_{L_f}$ be the solution to the Neumann problem with the boundary data $L_f$. In order to show that 
$\Tr(u_{L_f})=f$, it suffices to show that we have $\Tr(u_f)=\Tr(u_{L_f})$ in $HB^\theta_{p,p}(\partial\Omega)$, where 
$u_f\in HN^{1,p}(\Om)$ is the solution to the Dirichlet problem with boundary data $f$. 

For any $v\in HN^{1,p}(\Omega)$, since $u_f$ is $p$-harmonic and $\Tr(E(\Tr(v)))=\Tr(v)$, we have that
\[
L_f(\Tr(v))=\int_{\Omega}\!|\nabla u_f|^{p-2} \,\langle \nabla u_f, \nabla E\circ\Tr(v) \rangle\,d\mu 
= \int_{\Omega}\!|\nabla u_f|^{p-2} \,\langle \nabla u_f, \nabla v \rangle\,d\mu.
\]
On the other hand, since $u_{L_f}$ is a solution to the Neumann problem with boundary data $L_f$, we have
\[
L_f(\Tr(v))=\int_{\Omega}\!|\nabla u_{L_f}|^{p-2} \,\langle \nabla u_{L_f}, \nabla v \rangle\,d\mu.
\]
Hence, setting $v=u_{L_f}-u_f$ and using linearity of the Cheeger gradient, 
\[
\int_{\Omega}\!|\nabla u_{L_f}|^{p-2} \,\langle \nabla u_{L_f}, \nabla u_{L_f}-\nabla u_f \rangle\,d\mu 
= \int_{\Omega}\!|\nabla u_f|^{p-2} \,\langle \nabla u_f, \nabla  u_{L_f}-\nabla u_f \rangle\,d\mu.
\]
From here, it follows that 
\[
\int_{\Omega}\!\left\langle|\nabla u_{L_f}|^{p-2}\nabla u_{L_f}-|\nabla u_f|^{p-2}\nabla u_f, \nabla u_{L_f}-\nabla u_f \right\rangle d\mu =0,
\]
and so $\nabla u_{L_f}-\nabla u_f=0$ $\mu$-a.e. in $\Om$, that is, $u_{L_f}-u_f\in HN_0^{1,p}(\Om)$. 
Therefore, 
$\Tr u_{L_f}=\Tr u_f$ as members of the space $HB^\theta_{p,p}(\partial\Omega)$.
The claim about the norm of the Dirichlet-to-Neumann transform operator $\DtN$ follows now from
combining Corollary~\ref{cor:DtN-norm} with Corollary~\ref{cor:NtD-norm}.
\end{proof}

\bibliographystyle{amsplain}

\end{document}